\documentclass[a4paper,12pt]{article}
\usepackage[utf8]{inputenc}
\usepackage{amsmath,amssymb,amsthm}
\usepackage{graphics,longtable}
\usepackage[usenames]{color}
\usepackage{colortbl,alltt}
\usepackage{caption}
\usepackage{float}
\usepackage{arydshln}
\usepackage{nccrules}
\usepackage[mathscr]{eucal}
\usepackage[all]{xy}
\usepackage{url}

\newtheorem{defi}{Definition}[section]

\newtheorem{lema}[defi]{Lemma}
\newtheorem{prop}[defi]{Proposition}
\newtheorem{coro}[defi]{Corollary}

\newtheorem{thm}[defi]{Theorem}

\newcommand{\N}{\mathbb N}

\title{The right-generators descendant\\ 
of a numerical semigroup}
\author{Maria Bras-Amorós, Julio Fernández-González}
\date{\today}

\begin{document}
\maketitle
\noindent\let\thefootnote\relax\footnotetext{\hspace{-6.5truemm}\scriptsize 
The first author was supported by the Spanish government under grant TIN2016-80250-R and by the Catalan government under grant 2014 SGR 537.\\
The second author is partially supported by the Spanish government under grant MTM2015-66180-R.}

\begin{abstract}
\noindent For a numerical semigroup, we encode the set of primitive elements that are larger than its Frobenius number
and show how to produce in a fast way the corresponding sets for its children in the semigroup tree.
This allows us to present an efficient algorithm for exploring the tree up to a given genus.
The algorithm exploits the second nonzero element of a numerical semigroup and
the particular \emph{pseudo-ordinary} case in which this element is the conductor.
\end{abstract}

\section{\large Introduction}

\noindent Let $\Lambda$ be a \emph{numerical semigroup}, that is, an additive submonoid with finite complement 
in the ordered set $\N_0$ of nonnegative integers.
The \emph{gaps} of $\Lambda$ are the elements in the set $\N_0\setminus\Lambda$,
and so the elements of $\Lambda$ are also called \emph{nongaps}.
The number~$g$ of gaps is the \emph{genus} of $\Lambda$. We assume \,$g\neq 0$,
which amounts to saying that $\Lambda$ is not the trivial semigroup $\N_0$.

\vspace{1truemm}

Let \,$c$\, denote the \emph{conductor} of $\Lambda$, namely the least upper bound in $\Lambda$ of the set of gaps.
Thus, \,$c-1$\, is the largest gap of the semigroup, which is known as its \emph{Frobenius number}.
We refer to the nonzero nongaps that are smaller than the Frobenius number as \emph{left elements}.

\vspace{1truemm}

The lowest nonzero nongap~\,$m$\, is the \emph{multiplicity} of $\Lambda$.
We say that~$\Lambda$~is \emph{ordinary}\, whenever all its gaps are smaller than \,$m$, that is,
whenever \,$c=m$. This is also equivalent to asking the genus and the Frobenius number of~$\Lambda$ to be equal.
In other words, $\Lambda$ is ordinary if and only if it has no left elements. 
Otherwise, the set of left elements of $\Lambda$ is \\[-10pt]
$$\Lambda\cap\left\{m,\, m+1,\, \dots,\, c-2\right\}$$
and, in particular, it always contains the multiplicity $m$.

\vspace{1truemm}

A \emph{primitive element} of $\Lambda$ is a nongap \,$\sigma$\,
such that $\Lambda\!\setminus\!\{\sigma\}$ is still a semigroup,
that is, a nonzero nongap that is not the sum of two nonzero nongaps.
Primi\-tive elements constitute precisely the minimal generating set of the semigroup.
We use the term \emph{right generators} to refer to those primitive elements that
are not left elements. As it comes rightaway from the definitions, $\Lambda$
has at most $m$ right generators, since they must lie in the set \\[-10pt]
$$\left\{c,\, c+1,\, \dots,\, c+m-1\right\}\!.$$
The lowest primitive element of $\Lambda$ is \,$m$. It is a right generator
if and only if $\Lambda$ is ordinary; in this case, $\Lambda$~has exactly $m$ right generators.

\vspace{1truemm}

Any numerical semigroup of genus $g\geq 2$ can be uniquely obtained from
a semigroup of genus~\,$g-1$\, by removing a right generator. This idea gives rise to 
the construction of the semigroup tree, which was first introduced in \cite{oversemigroups}, \cite{fundamentalgaps} 
and also considered in \cite{Br:fibonacci}, \cite{Br:bounds}, \cite{BrBu}.
The nodes of this tree at level $g$ are all numerical semigroups of genus~$g$. The children, if any,
of a given node arise through the removals of each of its right generators.
The first levels are shown in Figure~\ref{tree}, where each semigroup is represented by its set of gaps,
as in \cite{gapsets}.

\vspace{1truemm}

\begin{figure}[H]
\centering 
\resizebox{\textwidth}{!}{
\includegraphics{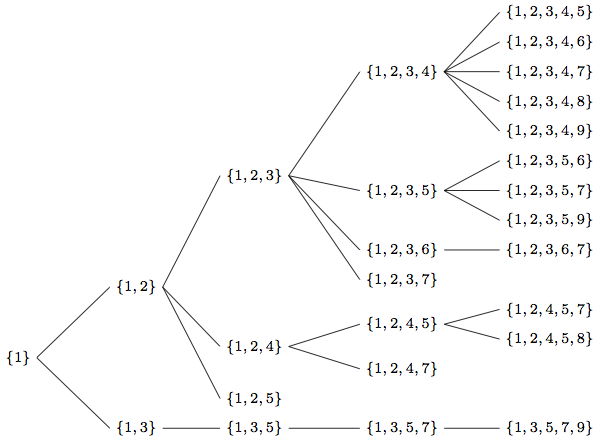}}
\caption{
}\label{tree}
\end{figure}

\vspace{1truemm}

\begin{lema}\label{heritage}
Let \,$\tilde\Lambda$ be the child obtained by removing a right generator \,$\sigma$ from~$\Lambda$. 
Then, the primitive elements of \,$\Lambda$ that become right generators of 
\,$\tilde\Lambda$ are those which are larger than $\sigma$.
\end{lema}

\begin{proof}
Since \,$\sigma\geq c$,\, the semigroup $\tilde\Lambda$ has Frobenius number $\sigma$.
Moreover, any primitive element of $\Lambda$ other than $\sigma$ is clearly a primitive element of $\tilde\Lambda$ too.
\end{proof}

\vspace{1truemm}

If follows from Lemma \ref{heritage} that, when successively removing the right gene\-ra\-tors of $\Lambda$ in increasing order, one obtains a numerical semigroup at each step. This leads to the main concept around which this paper is built.

\begin{defi}\label{RGD}{\rm The \emph{right-generators descendant} (RGD) of $\Lambda$ is the 
numerical semigroup \,${\cal D}(\Lambda)$ obtained by removing from $\Lambda$ all its right generators.
}
\end{defi}

\noindent Notice that \,${\cal D}(\Lambda)$\, is just $\Lambda$ precisely when $\Lambda$ 
does not have any right generators. Otherwise, it is a descendant of $\Lambda$ in the semigroup tree at some level. 

\vspace{1truemm}

In the next section we show how to compute efficiently from \,${\cal D}(\Lambda)$\, the RGD of the children
of $\Lambda$. To this end, we restate here a special case of Lemma~2.1 in~\cite{seeds},
where right generators correspond to \emph{order-zero seeds}. The result complements Lemma~\ref{heritage}
by dealing with \,\emph{new right generators}, that is, with the right generators of a child
that are not primitive elements of its parent. The key point goes back to Lemma~3 in~\cite{Br:bounds}.
We provide a short proof whose content is partially needed later on.

\vspace{1truemm}

\begin{lema}\label{new}
Let \,$\tilde\Lambda$ be the child obtained by removing a right generator \,$\sigma$ from~$\Lambda$.
The semigroup $\tilde\Lambda$ is ordinary if and only if \,$\Lambda$ is ordinary and \,$\sigma=m$.
Otherwise, the only possible new right generator of \,$\tilde\Lambda$ could be \,$\sigma+m$.
\end{lema}

\begin{proof}
The conductor of $\tilde\Lambda$ is \,$\sigma+1$.
If \,$\sigma=m$, which can only occur when $\Lambda$ is ordinary, 
then $\tilde\Lambda$ has multiplicity $m+1$ and it is also ordinary.
If \,$\sigma\neq m$, then $\tilde\Lambda$ has still multiplicity $m$ so it cannot be ordinary. 
Moreover, new right generators of $\tilde\Lambda$ must be of the form 
\,$\sigma+\lambda$ for some $\lambda\geq m$,
so the result follows from the fact that any primitive element must be smaller than
the sum of the conductor and the multiplicity.
\end{proof}

\begin{coro} 
The semigroup $\Lambda$ is ordinary if and only if \,${\cal D}(\Lambda)$ is ordinary.
In this case, the conductor of \,${\cal D}(\Lambda)$ is twice the conductor of \,$\Lambda$.
\end{coro}

\begin{proof}
If \,$\sigma_1,\dots,\sigma_k$\, are the right generators of \,$\Lambda$\, in increasing order, then
Lem\-ma~\ref{heritage} allows us to construct \,${\cal D}(\Lambda)$ through the chain of semigroups
$$\Lambda_0\,=\,\Lambda \,\supset\, \Lambda_1  \,\supset\, \cdots \cdots \,\supset\, 
\Lambda_k \,=\, {\cal D}(\Lambda),$$
where $\Lambda_j$ is obtained by removing $\sigma_j$ from $\Lambda_{j-1}$ for $j=1,\dots,k$.
By Lemma~\ref{new}, $\Lambda_j$ is ordinary if and only if \,$\Lambda_{j-1}$\, is
ordinary, since \,$\sigma_j$\, is the lowest right generator of $\Lambda_{j-1}$. This
implies the first part of the statement. The second part follows from 
the fact that \,$c,\, c+1,\,\dots\,, 2c-1$\, are the primitive elements of the ordinary semigroup with conductor $c$.
\end{proof}

\vspace{3truemm}

The main goal of this article is an algorithm to explore the semigroup tree which is based
on a binary encoding of the RGD. A preliminary version is given in Section 2, and then
the algorithm is presented in full form in Section 4. 
An important role is played by what we call the \emph{jump} of a numerical semigroup,
which is the difference between its first two nonzero elements.
This parameter is exploited in Section 3 for our purposes, leading to 
a structure of the algorithm for which the subtrees of semigroups
with the same multiplicity and jump can be explored in parallel.
In Section 5, we test the efficiency of the algorithm by applying it 
to the computational problem of counting numerical semigroups by their genus.
Our running times are shorter than those required by the best known algorithms.

\vspace{1truemm}

\section{\large Binary encoding of the RGD}

Let us fix a numerical semigroup $\Lambda$ with multiplicity $m$ and conductor $c$.
We associate with $\Lambda$ the binary string
$$D(\Lambda)\,=\,D_0\,D_1\,\cdots\,D_j\,\cdots $$
encoding its RGD as follows: for $j\geq 0$,
$$D_j:=\left\{\begin{array}{l}
{\bf 0} \quad \mathrm{if} \ \ m+j\,\in\,{\cal D}(\Lambda),\\[5pt]
{\bf 1} \quad \mathrm{otherwise}.
\end{array}
\right. 
$$

\vspace{1truemm}

The left elements of $\Lambda$ can be read from the first \,$c-m$\, bits of $D(\Lambda)$. 
Indeed, for \,$0\leq j<c-m$, one has \,$D_j={\bf 0}$\, if and only if \,$m+j$\, is not a gap of $\Lambda$.
In particular, this piece of the string $D(\Lambda)$, together with the multiplicity $m$, determines uniquely
the semigroup $\Lambda$.

\vspace{1truemm}

As for the rest of the string, only the first \,$m$\, bits may take the value ${\bf 1}$, 
depending on whether they correspond to right generators of $\Lambda$ or not:
for \mbox{\,$j\geq c-m$,} one has \,$D_j={\bf 1}$\, if and only if \,$m+j$\, is a primitive element of $\Lambda$.
In particular, $\Lambda$~is ordinary if and only if \,$D_0={\bf 1}${. In this case,} \,$D_j={\bf 1}$\, for \,$j<m$.

\vspace{1truemm}

Since \,$D_j={\bf 0}$\, for \,$j\geq c$, we can identify $D(\Lambda)$ 
with the chain consisting of its first~\,$c$\, bits:
$$D(\Lambda)\,=\,D_0\,D_1\,\cdots\,D_{c-1}.$$
The binary chains \,$D(\Lambda)$,\, for $\Lambda$ running over the first levels
of the semigroup tree, are displayed in Figure \ref{treeDL}, where the bits encoding right 
generators are highlighted in grey. The vertical dashed line at each chain lies just before the entry
corresponding to the conductor of the semigroup. The number of bits after that line is equal to the multiplicity.

\vspace{-1truemm}

\begin{figure}[H]
\centering
\resizebox{\textwidth}{!}{
\includegraphics{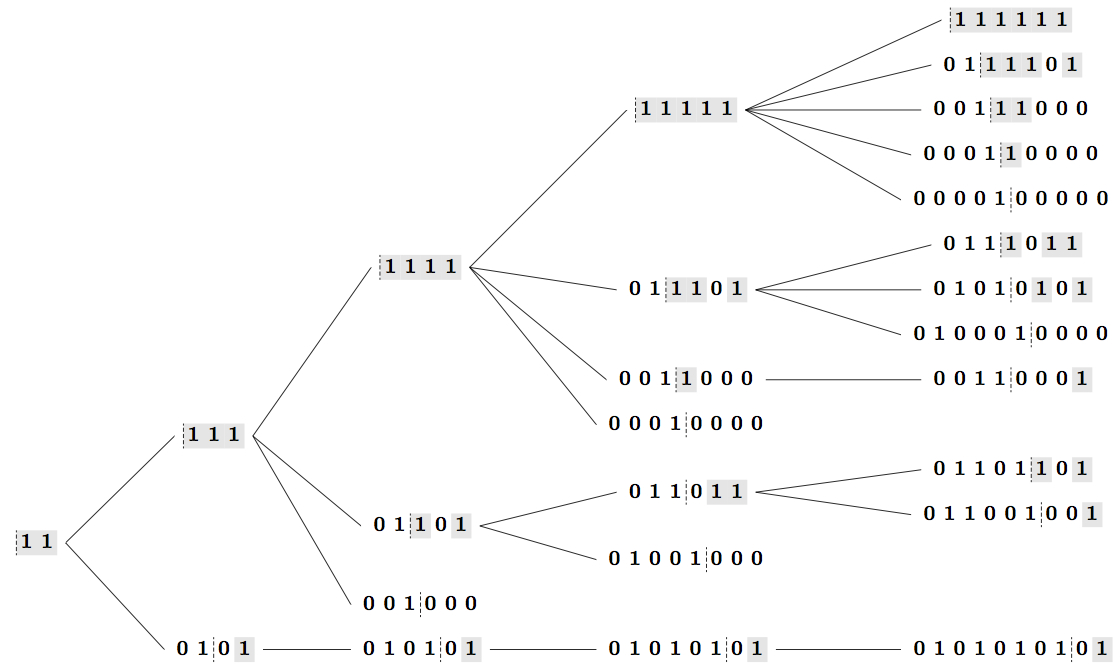}}
\caption{
}\label{treeDL}
\end{figure}

\vspace{1truemm}

\begin{lema}\label{ordinary}
If \,$\Lambda$ is ordinary, then it has exactly \,$m$ children, namely
$$\tilde\Lambda_s=\Lambda\setminus\{m+s\} \qquad \mathit{for} \qquad 0\leq s<m.$$
The child \,$\tilde\Lambda_0$ is the ordinary semigroup of multiplicity \,$m+1$, which means that
the bit chain associated with this child is
$$D(\tilde\Lambda_0) \,=\, \dashrule[-0.6ex]{0.4}{3 1.5 3 1.5 3}\, 
\underbrace{{\bf 1} \ {\bf 1} \,\cdots\cdots\, {\bf 1}}_{m+1}.$$
The other children have multiplicity \,$m$, and only $\tilde\Lambda_1$ has a new right generator,
namely \,$2m+1$. More precisely, the bit chains associated with them are as follows:
$$D(\tilde\Lambda_1) \,=\, {\bf 0} \ {\bf 1} \,\,\dashrule[-0.6ex]{0.4}{3 1.5 3 1.5 3}\, 
\underbrace{{\bf 1} \,\cdots\, {\bf 1}}_{m-2} \ {\bf 0} \ {\bf 1}\,,$$
whereas, assuming \,$m>2$,
$$D(\tilde\Lambda_s) \,=\, \underbrace{{\bf 0} \,\cdots\, {\bf 0}}_s \ {\bf 1}
\,\,\dashrule[-0.6ex]{0.4}{3 1.5 3 1.5 3}\,
\underbrace{{\bf 1} \,\cdots\, {\bf 1}}_{m-s-1} \, \underbrace{{\bf 0} \, \cdots \, {\bf 0}}_{s+1}
\qquad \mathit{for} \qquad 2\leq s\leq m-1.$$
\end{lema}

\vspace{1truemm}

\begin{proof}
The right generators of $\Lambda$ are \,$m,\, m+1,\, \dots,\, 2m-1$.
By Lemma~\ref{new}, $\tilde\Lambda_s$ is not ordinary if and only if $s\neq 0$, and then
the only possible new right generator of \,$\tilde\Lambda$ could be \,$2m+s$.
Since $\tilde\Lambda_1$ has multiplicity \,$m$\, and conductor \,$m+2$, the integer \,$2m+1$\, is
a primitive element of this child. By contrast, $2m+s$\, is not a primitive element of $\tilde\Lambda_s$
if $s>1$, because \,$m+1$\, and \,$m+s-1$\, are both nongaps of this child.
The result follows then from  Lemma~\ref{heritage} and the definition of the bit chain encoding the RGD
of a numerical semigroup.
\end{proof}

\vspace{1truemm}

The children $\tilde\Lambda_1,\dots,\tilde\Lambda_{m-1}$ in Lemma \ref{ordinary} are the  
numerical semigroups of multiplicity $m$ having genus $m$. We refer to them as \emph{quasi-ordinary} semigroups.

\vspace{1truemm}

Let us go back to the general case and consider a child $\tilde\Lambda$, 
which is obtained by removing a right generator~$\sigma$ from~$\Lambda$.
The index \,$s$\, corresponding to \,$\sigma$\, in the bit chain $D(\Lambda)$, namely \,$s=\sigma-m$,\,
satisfies \,$c-m\leq s<c$.
The condition \,$D_s={\bf 1}$\, is equivalent to~\,$\sigma$\, being a right generator of $\Lambda$.
Notice that the Frobenius number of $\tilde\Lambda$ is \,$m+s$.

\vspace{1truemm}

\begin{lema}\label{newchain}
Let \,$D(\tilde\Lambda)=\tilde D_0\,\tilde D_1\,\cdots\,\tilde D_{m+s}$\, 
be the bit chain associated with $\tilde\Lambda$. Then,
$$\widetilde D_j \,=\, \left\{\begin{array}{ccl}
D_j & \quad \mathrm{for} \quad\quad  & 0\,\leq\, j\,<\,c\,-\,m,\\[10pt]
\bf 0 & \quad \mathrm{for} \quad\quad  & c\,-\,m\,\leq\, j\,<\,s,\\[10pt]
D_j & \quad \mathrm{for} \quad\quad  & s\,\leq\, j\,<\,c,\\[10pt]
{\bf 0} & \quad \mathrm{for} \quad\quad  & c\,\leq\, j\,<\,m+s.
\end{array}\right.$$
Furthermore, $\tilde D_{m+s}={\bf 0}$ if and only if there is a positive integer \,$\ell\leq\lfloor s/2\rfloor$\, such that \,$\tilde D_{\ell}=\tilde D_{s-\ell}={\bf 0}$. This amounts to saying that
$\tilde\Lambda$ does not have a new right generator.
\end{lema}

\begin{proof} 
The statements hold if $\Lambda$ is ordinary, as can be checked from Lemma~\ref{ordinary}.
So we can assume \,$c\neq m$, hence \,$s\neq 0$\, and $\tilde\Lambda$ has multiplicity \,$m$.
The result comes from the definitions of $D(\Lambda)$ and $D(\tilde\Lambda)$ along with the
following remarks. The left elements of~$\tilde\Lambda$ are those of $\Lambda$
together with the integers $\lambda$ such that \,$c\leq\lambda<m+s$, which
yields the values of $\widetilde D_j$ for $j<s$, whereas \,$\tilde D_{s}={\bf 1}=D_{s}$\,
because $m+s$ is a gap of $\tilde\Lambda$.
As for the right generators of $\tilde\Lambda$, according to Lemma~\ref{heritage} and Lemma~\ref{new},
they are those which are inherited from $\Lambda$, hence above \,$m+s$\, and below \,$c+m$, and maybe also \,$2m+s$.
This candidate to new right generator is not a primitive element of $\tilde\Lambda$
if and only if there is a positive integer~\,$\ell$\, such that \,$m+\ell$\, and \,$m+s-\ell$\, are
nongaps of~$\tilde\Lambda$. Since the midpoint of these two values is \,$m+s/2$, such an integer
can be assumed to satisfy \,$\ell\leq\lfloor s/2\rfloor$. Notice that \,$m+\ell$\, and \,$m+s-\ell$\,
are below the Frobenius number \,$m+s$, so they lie in $\tilde\Lambda$ if and only if 
\,$\tilde D_{\ell}=\tilde D_{s-\ell}={\bf 0}$.
\end{proof}

\vspace{1truemm}

Lemma~\ref{newchain} leads to a procedure to explore the semigroup tree using the RGD binary encoding. 
More precisely, the following recursive function visits the 
descendants of the semigroup $\Lambda$ up to a fixed genus $\gamma$. 
The first input item stands for the binary string associated with $\Lambda$, and $g$ stands for the genus. 
If $\Lambda$ is not ordinary, line 7 can be deleted, 
and then \,$\tilde m$\, should be replaced with \,$m$\, at line~12.
An implementation of this procedure runs on the website \cite{combos}.

\vspace{5truemm}

{\footnotesize
\hrule
\hrule
\vspace{2truemm}
\noindent{\quad \  $\mathbf{RGD-offspring}_\gamma(D,\,m,\,c,\,g)$}\\[-5pt]
\hrule
\begin{enumerate}\itemsep0.6truemm
\item \ {\bf if} \ \,$g<\gamma$\, \ {\bf then}
\item \qquad $\tilde D:=D$
\item \qquad $\tilde g:=g+1$
\item \qquad {\bf for} \ \,$s$\, \ {\bf from} \ \,$c-m$\, \ {\bf to} \ \,$c-1$\, \ {\bf do}
\item \qquad\qquad{\bf if} \ \,$\tilde D_s \,\neq\, {\bf 0}$\, \ {\bf then}
\item \qquad\qquad\qquad $\mathrm{\texttt Frob}:=m+s$
\item \qquad\qquad\qquad {\bf if} \ \,$s=0$\, \ {\bf then} \ \,$\tilde m:=m+1$\, \ {\bf else} \ \,$\tilde m:=m$
\item \qquad\qquad\qquad $\ell:=1$
\item \qquad\qquad\qquad $\mathrm{\texttt mid}:=\lfloor s/2\rfloor$
\item \qquad\qquad\qquad{\bf while} \ \,$\ell\leq\mathrm{\texttt mid}$\, 
\ {\bf and} \ $\big(\tilde D_{\ell}\neq{\bf 0}\,$ \ \,{\bf or}\, \ $\tilde D_{s-\ell}\neq{\bf 0}\big)$
\ {\bf do} \ \,$\ell:=\ell + 1$
\item \qquad\qquad\qquad {\bf if} \ \,$\ell>\mathrm{\texttt mid}$\, \ {\bf then} 
\ \,$\tilde D_{\mathrm{\texttt Frob}}:={\bf 1}$\,
\item \qquad\qquad\qquad $\mathbf{RGD-offspring}_\gamma\big(\tilde D, \ \tilde m, \ \mathrm{\texttt Frob}+1, \ \tilde g\big)$
\item \qquad\qquad\qquad $\tilde D_{\mathrm{\texttt Frob}} := {\bf 0}$
\item \qquad\qquad\qquad $\tilde D_s := {\bf 0}$
\end{enumerate}
\hrule
\hrule
}

\vspace{5truemm}

When running this function, the most expensive computation for each descendant is at worst 
the checking loop at line 10. The number of steps in the loop is bounded by half the conductor of the node, 
hence by the genus.
And each step carries out at most five basic operations, which are constant or logarithmic on the genus. 
This gives the following estimate for the time complexity.

\begin{prop}
The number of bit operations to explore the semigroup tree up to genus $\gamma$
using the \,$\mathbf{RGD-offspring}_\gamma$\, function is
$$O\bigg(\gamma\,\log\gamma\,\sum_{g\,\leq\,\gamma}\!n_g\bigg),$$\\[-8pt]
where \,$n_g$ stands for the number of numerical semigroups of genus \,$g$.
\end{prop}


\section{\large Pseudo-ordinary semigroups}

As in the previous section, let us fix an integer $m\geq 2$, and 
let $\Lambda$ stand for a numerical semigroup with multiplicity $m$. Let then
\,$m+u$\, be the second nonzero element of $\Lambda$, that is, its first nongap above $m$. 
We refer to the positive integer~\,$u$\, as the \emph{jump} of $\Lambda$.
Obviously, it cannot be larger than~\,$m$, and the conductor \,$c$\, of $\Lambda$
is at least \,$m+u$\, unless $\Lambda$ is ordinary, in which case \,$u=1$.

\vspace{1truemm}

Let us first exploit the jump to improve Lemma~\ref{new}. 
According to the nomenclature introduced in~\cite{BrBu}
for the nonordinary case, we say that a right generator~\,$\sigma$ of~$\Lambda$ 
is a \emph{strong generator} if \,$\sigma+m$\, is a primitive element of 
the child~{$\Lambda\setminus\{\sigma\}$}. 
The next result gives a threshold for strong generators.
The ordinary case was already treated separately in Lemma~\ref{ordinary}.

\vspace{1truemm}

\begin{lema}\label{threshold}
Let \,$\tilde\Lambda$ be the child obtained by removing a right generator \,$\sigma$ from~$\Lambda$.
If \,$\Lambda$ is not ordinary and \,$\sigma\geq c+u$, then \,$\sigma$ satisfies the following:
\begin{itemize}
\item[$\cdot$] It is not a strong generator of \,$\Lambda$.
\item[$\cdot$] The only strong generator of \,$\tilde\Lambda$ could be~\,$\sigma + u$, in which case \,$2u<m$.
\end{itemize}
\end{lema}

\begin{proof}
For an integer $\lambda\geq\sigma$\, other than \,$\sigma+u$, both \,$\lambda-u$\, and \,$m+u$\, are nonzero elements
of the set $\tilde\Lambda\setminus\{\lambda\}$ because \,$u$\, is positive and
$$m+u\,\leq\,c\,\leq\,\sigma-u\,\leq\lambda-u\,\neq\,\sigma.$$
Since \,$\lambda + m$\, is the sum of \,$\lambda-u$\, and  \,$m+u$,\, it cannot be a primitive element
of $\tilde\Lambda\setminus\{\lambda\}$ whenever this set is a numerical semigroup. The case \,$\lambda=\sigma$\, 
amounts to the first item of the statement. As for the second, what remains to be proved is the inequality.
So let us assume that \,$\sigma+u$\, is a primitive element of $\tilde\Lambda$.
By the first item and Lemma~\ref{new}, it is then necessarily a right generator of $\Lambda$ and, in particular,
below \,$c+m$. The hypothesis \,$\sigma\geq c+u$\, implies then \,$2u<m$.
\end{proof}

\vspace{1truemm}

\begin{defi}{\rm We say that~$\Lambda$~is \,\emph{pseudo-ordinary}\, whenever \,$m+u$\, 
is precisely the conductor \,$c$, that is, whenever \,$c-2$\, is the genus of $\Lambda$.
This condition amounts to asking \,$m$\, to be the only left element of $\Lambda$.
In this case, the jump \,$u$\, is at least~$2$.
}
\end{defi}

\vspace{1truemm}

For an integer \,$u$\, with \,$2\leq u\leq m$, let us use the notation $\Lambda_u$ 
for the pseudo-ordinary semigroup of multiplicity \,$m$\, and jump \,$u$. 
The conductor of $\Lambda_u$ is \,$m+u$. 
The set of right generators of $\Lambda_u$ is then
$$\left\{m+u,\, m+u+1,\, \dots\dots,\, 2m+u-1\right\} \setminus \{2m\}.$$
In other words, the bit chain associated with $\Lambda_u$ is
$$D(\Lambda_u) \,=\, {\bf 0} \ \underbrace{{\bf 1} \,\cdots\, {\bf 1}}_{u-1} 
\,\,\dashrule[-0.6ex]{0.4}{3 1.5 3 1.5 3}\,
\underbrace{{\bf 1} \,\cdots\, {\bf 1}}_{m-u} \ {\bf 0} \ \underbrace{{\bf 1} \,\cdots\, {\bf 1}}_{u-1}.$$
Notice that $\Lambda_2$ is the only pseudo-ordinary semigroup of multiplicity \,$m$\, that is also
quasi-ordinary, that is, a nonordinary child of an ordinary semigroup. Specifically, it is the child
$\tilde\Lambda_1$ in Lemma \ref{ordinary}.

\vspace{1truemm}

The next result complements Lemma \ref{threshold} for pseudo-ordinary semigroups
by showing that the first \,$u$\, or \,$u-1$\, right generators of $\Lambda_u$, 
depending on whether \,$2u\leq m$\, or not, are strong.

\vspace{1truemm}

\begin{lema}\label{pseudo}
Let \,$s\neq m$ be an integer satisfying \,$u\leq s<2u$. Then,
\,$m+s$\, is a strong generator of the pseudo-ordinary semigroup $\Lambda_u$.
\end{lema}

\begin{proof}
Let $\tilde\Lambda_{u,\,s}$ be the child of $\Lambda_u$ obtained by removing 
the right generator~\,$m+s$.
Take a nongap of $\tilde\Lambda_{u,\,s}$ above \,$m$, that is, an integer $\lambda$ such that 
$$m+u \,\leq\, \lambda \,\neq\, m+s.$$\\[-20pt]
Then,
$$m \,\neq\, 2m+s-\lambda \,\leq\, s + m - u \,<\, m+u,$$
hence \,$2m+s-\lambda$\, cannot be a nonzero element of $\Lambda_u$.
So \,$2m+s$\, is a primitive element of $\tilde\Lambda_{u,\,s}$ above the Frobenius number \,$m+s$.
\end{proof}

\vspace{1truemm}

Notice that, with the notation in the proof of Lemma \ref{pseudo}, one has 
$$\tilde\Lambda_{u,\,u}=\Lambda_{u+1} \qquad \mathrm{for} \qquad   2\leq u\leq m-1.$$
In particular, the pseudo-ordinary semigroups of multiplicity \,$m$
are the nodes of a path in the semigroup tree starting at level \,$m$\, and ending at level \,$2m-2$\,:
$$\xymatrix{
\Lambda_2 \ar@{-}[rr] & & \Lambda_3 \ar@{-}[rr] && \cdots\cdots \ar@{-}[rr] & & \Lambda_m
}
$$

\vspace{5truemm}

We now go back again to the general case in order to characterize 
pseudo-ordinary descendants in the semigroup tree in terms of the monotony of the jump. 
Let \,$\tilde u$\, denote the jump of a child $\tilde\Lambda$ of $\Lambda$.
As shown in Lemma \ref{new}, $\tilde\Lambda$~inherits the same multiplicity $m$ if and only if it is 
not ordinary, and then the jump \,$\tilde u$\, is at most \,$m$.\,
If $\tilde\Lambda$, hence also $\Lambda$, is ordinary, then \,$\tilde u = u = 1$.

\vspace{1truemm}

\begin{prop}\label{newu}
Let \,$\tilde\Lambda$ be the child obtained by removing a right generator \,$\sigma$ from~$\Lambda$.
Then, $\tilde\Lambda$  is not pseudo-ordinary if and only if \,$\tilde u=u$.
More precisely, $\tilde\Lambda$~is pseudo-ordinary, and then \,$\tilde u=u+1$, exactly in one of these two cases:
\begin{itemize}
\item[$\cdot$] If \,$\Lambda$ is ordinary and \,$\sigma=m+1$. 
\item[$\cdot$] If \,$\Lambda$ is pseudo-ordinary and \,$\sigma=m+u$. In this case, \,$2\leq u<m$.
\end{itemize}
With the above notation, $\tilde\Lambda=\Lambda_2$ in the first item,
whereas \,$\Lambda=\Lambda_u$ and~\,\mbox{$\tilde\Lambda=\Lambda_{u+1}$} in the second one.
\end{prop}

\begin{proof}
We can assume that $\tilde\Lambda$ has multiplicity $m$. Its second nonzero element is then~\,$m+\tilde u$.
The parent $\Lambda$ is obtained from~$\tilde\Lambda$ by putting the Frobenius number~\,$\sigma$\, back.
Thus, if \,$m+\tilde u$\, is a left element of~$\tilde\Lambda$, then
it is necessarily equal to \,$m+u$\, because \,$\sigma$\, is larger.
Conversely, assume that~\,$m+\tilde u$\, is the conductor of $\tilde\Lambda$, which implies
\,$\sigma=m+\tilde u-1$\, and~\,$\tilde u\geq 2$.
If~\,$\tilde u=2$, then \,$\sigma$\, is the only gap of~$\tilde\Lambda$ larger than \,$m$,
hence $\Lambda$ is ordinary. If~\,$\tilde u>2$, then \,$m$\, is still the only left element of $\Lambda$
and \,$\sigma$\, is its conductor, hence \,$\sigma=m+u$\, and the result follows.
\end{proof}

\vspace{1truemm}

\begin{coro}
The second nonzero element of \,$\Lambda$ is \,$m+1$ if and only if \,$\Lambda$ is not pseudo-ordinary 
nor a descendant of a pseudo-ordinary semigroup.
\end{coro}

\vspace{1truemm}

\section{\large The RGD algorithm}

In this section we produce an algorithm to explore the semigroup tree. This is done 
by putting together the contents that have been developed so far.

\vspace{1truemm}

Let us start by exploiting the first item of Lemma~\ref{threshold} to present a refined version 
of the recursive procedure below Lemma \ref{newchain}.
The input corresponds now to a numerical semigroup that is not ordinary nor pseudo-ordinary.
In order to shorten the checking loop at line 10 of the $\mathbf{RGD-offspring}_\gamma$ function, 
as well as the main loop, and hence reduce the number of times that the former is executed, 
it incorporates as input parameters both the jump~\,$u$\, and the number~\,$r$\, of right generators. 
The semigroup is \emph{handled} in a concrete prescribed way: the corresponding parameters may be printed, 
for instance, or a certain property may be checked. This is the purpose of the $\mathbf{handle}$
function at the first line.

\vspace{10truemm}

{\footnotesize
\hrule
\hrule
\vspace{2truemm}
\noindent{\quad \  $\mathbf{RGD}_\gamma(D, \ m, \ u, \ c, \ g, \ r)$}\\[-5pt]
\hrule
\begin{enumerate}\itemsep0.9truemm
\item \ {\bf handle}\,$(D, \ m, \ c)$
\item \ {\bf if} \ \,$g<\gamma$\, \ {\bf then}
\item \qquad $\tilde D:=D$
\item \qquad $\tilde r:=r$
\item \qquad {\bf for} \ \,$s$\, \ {\bf from} \ \,$c-m$\, \ {\bf to} \ \,$c-m+u-1$\, \ {\bf do}
\item \qquad\qquad{\bf if} \ \,$\tilde D_s \,\neq\, {\bf 0}$\, \ {\bf then}
\item \qquad\qquad\qquad $\mathrm{\texttt Frob}:=m+s$
\item \qquad\qquad\qquad $\ell:=u$
\item \qquad\qquad\qquad $\mathrm{\texttt mid}:=\lfloor s/2\rfloor$
\item \qquad\qquad\qquad{\bf while} \ \,$\ell\leq\mathrm{\texttt mid}$\, 
\ {\bf and} \ $\big(\tilde D_{\ell}\neq{\bf 0}\,$ \ \,{\bf or}\, \ $\tilde D_{s-\ell}\neq{\bf 0}\big)$
\ {\bf do} \ \,$\ell:=\ell + 1$
\item \qquad\qquad\qquad {\bf if} \ \,$\ell>\mathrm{\texttt mid}$\, \ {\bf then}
\item \qquad\qquad\qquad\qquad $\tilde D_{\mathrm{\texttt Frob}} := {\bf 1}$\,
\item \qquad\qquad\qquad\qquad $\mathbf{RGD}_\gamma\big(\tilde D, \ m, \ u, \ \mathrm{\texttt Frob}+1, \ g+1, \ \tilde r\big)$
\item \qquad\qquad\qquad\qquad $\tilde r := \tilde r-1$
\item \qquad\qquad\qquad\qquad $\tilde D_{\mathrm{\texttt Frob}} := {\bf 0}$
\item \qquad\qquad\qquad {\bf else} \, {\bf then}
\item \qquad\qquad\qquad\qquad $\tilde r := \tilde r-1$
\item \qquad\qquad\qquad\qquad $\mathbf{RGD}_\gamma\big(\tilde D, \ m, \ u, \ \mathrm{\texttt Frob}+1, \ g+1, \ \tilde r\big)$
\item \qquad\qquad\qquad $\tilde D_s := {\bf 0}$
\item \qquad{\bf while} \ \,$\tilde r>0$\, \ {\bf do}
\item \qquad\qquad $s := s+1$
\item \qquad\qquad {\bf if} \ \,$\tilde D_s \,\neq\, {\bf 0}$\, \ {\bf then}
\item \qquad\qquad\qquad $\tilde r := \tilde r-1$
\item \qquad\qquad\qquad $\mathbf{RGD}_\gamma\big(\tilde D, \ m, \ u, \ m+s+1, \ g+1, \ \tilde r\big)$
\item \qquad\qquad\qquad $\tilde D_s := {\bf 0}$
\end{enumerate}
\hrule
\hrule
}


\bigskip

For an integer \,$m\geq 2$, let ${\mathscr T}_m$ be the subtree of numerical semigroups
with multiplicity $m$, which is rooted at an ordinary semigroup. We present here an algorithm 
that sequentially explores \,${\mathscr T}_2, \dots,{\mathscr T}_{\gamma+1}$ 
up to a fixed level~$\gamma$ as detailed in the pseudocode below.
The first six lines perform the walk through the path~${\mathscr T}_2$.
Then, for \,$3\leq m\leq\gamma$, each subtree ${\mathscr T}_m$ is explored by depth first search.
The children of every node are visited by lexicographical order of their gaps.
Finally, the last three lines of the pseudocode deal with the ordinary semigroup of genus~$\gamma$.

\vspace{1truemm}

The binary string at line 8 corresponds to the ordinary semigroup of mul\-ti\-pli\-ci\-ty $m$.
The loop at line 12 runs along the path consisting of the pseudo-ordinary semigroups of multiplicity $m$ 
whose genus is at most $\gamma$, except for the last node of the path, which is treated separately at
lines~17 to 20.

\vspace{1truemm}

The loop at line 23 runs over the quasi-ordinary semigroups of multiplicity~$m$,
except for the pseudo-ordinary one. Specifically, the $\mathbf{RGD}_\gamma$ function is applied to
the children $\tilde\Lambda_2,\dots,\tilde\Lambda_{m-2}$ in Lemma~\ref{ordinary}.
The binary string associated with~$\tilde\Lambda_2$ is introduced at line~21.
The variable~\,$r$\, in the loop stores the number of right generators of the sibling
at each step. The child $\tilde\Lambda_{m-1}$ does not have any descendants 
and it is handled at line~27.

\vspace{10truemm}

{\footnotesize
\hrule
\hrule
\vspace{2truemm}
\noindent{\quad \  {\bf RGD \,algorithm}}\\[-5pt]
\hrule
\vspace{2truemm}
\noindent\quad \ {\bf Input: } \ level \ $\gamma\geq 2$\\[-5pt]
\hrule
\begin{enumerate}\itemsep0.8truemm
\item \ $D \,:=\, {\bf 1} \ {\bf 1} \ {\bf 0} \ {\bf 0} \ \cdots$
\item \ {\bf handle}\,$(D, \ 2, \ 2)$
\item \ $D_0 := {\bf 0}$
\item \ {\bf for} \ \,$g$\, \ {\bf from} \ \,$2$\, \ {\bf to} \ \,$\gamma$\, \ {\bf do}
\item \qquad $D_{2g-1} := {\bf 1}$
\item \qquad {\bf handle}\,$(D, \ 2, \ 2g)$
\item \ {\bf for} \ \,$m$\, \ {\bf from} \ \,$3$\, \ {\bf to} \ \,$\gamma$\, \ {\bf do}
\item \qquad $D \,:=\, \underbrace{{\bf 1} \,\cdots\, {\bf 1}}_m \ {\bf 0} \ {\bf 0} \ \cdots$
\item \qquad {\bf handle}\,$(D, \ m, \ m)$
\item \qquad $D_{0} := {\bf 0}$
\item \qquad $\mathrm{\texttt{path\_end}}:=\mathrm{min}\left(m, \, \gamma+2-m\right)$
\item \qquad {\bf for} \ \,$u$\, \ {\bf from} \ \,$2$\, \ {\bf to} \ \,$\,\mathrm{\texttt{path\_end}}-\,1$\, 
\ {\bf do}
\item \qquad\qquad $c := m+u$
\item \qquad\qquad $D_{c-1} := {\bf 1}$
\item \qquad\qquad {\bf handle}\,$(D, \ m, \ c)$
\item \qquad\qquad $\mathbf{pseudo}_\gamma(D, \ m, \ u, \ c, \ m-2)$
\item \qquad $c := m \,+\, \mathrm{\texttt{path\_end}}$
\item \qquad $D_{c-1} := {\bf 1}$
\item \qquad {\bf handle}\,$(D, \ m, \ c)$
\item \qquad {\bf if} \ \,$\mathrm{\texttt{path\_end}} < \gamma+2-m$\, \ {\bf then}  \ $\mathbf{pseudo}_\gamma(D, \ m, \ m, \ c, \ m-1)$\\[-5pt]
\item \qquad $D \,:=\, {\bf 0} \ {\bf 0} \ \underbrace{{\bf 1} \,\cdots\, {\bf 1}}_{m-2} \ {\bf 0} \ {\bf 0} \ \cdots$
\item \qquad $r := m-3$
\item \qquad {\bf for} \ \,$s$\, \ {\bf from} \ \,$2$\, \ {\bf to} \ \,$m-2$\, \ {\bf do}
\item \qquad\qquad $\mathbf{RGD}_\gamma(D, \ m, \ 1, \ m+s+1, \ m, \ r)$
\item \qquad\qquad $r := r-1$
\item \qquad\qquad $D_s := {\bf 0}$
\item \qquad {\bf handle}\,$(D, \ m, \ 2m)$
\item \ $m :=\gamma+1$
\item \ $D \,:=\, \underbrace{{\bf 1} \,\cdots\, {\bf 1}}_m \ {\bf 0} \ {\bf 0} \ \cdots$
\item \ {\bf handle}\,$(D, \ m, \ m)$
\end{enumerate}
\hrule
}

\vspace{10truemm}

The $\mathbf{pseudo}_\gamma$ function computes the children of a given pseudo-ordinary semigroup~$\Lambda$
that are not pseudo-ordinary, that is, not fulfilling the second item of Proposition~\ref{newu}, 
and it launches the exploration of their descendants in the tree.
The input parameters of~$\Lambda$ include the jump \,$u$, 
which is just the difference~\,$c-m$\, in this case.
The variable~\,$\tilde r$\, stores the number of right generators of the siblings.
The loop at line~3, which is a rightaway application of Lemma~\ref{pseudo}, 
runs over the strong generators of~{$\Lambda$}, that is, over 
the children of~$\Lambda$ having a new right generator.
By contrast, the loop at line~11 makes use of the first item in
Lemma~\ref{threshold} to deal with the rest of siblings.

\vspace{10truemm}

{\footnotesize
\hrule
\hrule
\vspace{2truemm}
\noindent{\quad \  $\mathbf{pseudo}_\gamma(D, \ m, \ u, \ c, \ \tilde r)$}\\[-5pt]
\hrule
\begin{enumerate}\itemsep0.8truemm
\item \ $\tilde D:=D$
\item \ $\tilde D_u := {\bf 0}$
\item \ {\bf for} \ \,$s$\, \ {\bf from} \ \,$u+1$\, \ {\bf to} \ \,$2u-1$\, \ {\bf do}
\item \qquad {\bf if} \ \,$s\neq m$\, \ {\bf then}
\item \qquad\qquad $\mathrm{\texttt Frob}:=m+s$
\item \qquad\qquad $\tilde D_{\mathrm{\texttt Frob}} := {\bf 1}$\,
\item \qquad\qquad $\mathbf{RGD}_\gamma\big(\tilde D, \ m, \ u, \ \mathrm{\texttt Frob}+1, \ c-1, \ \tilde r\big)$
\item \qquad\qquad $\tilde r := \tilde r-1$
\item \qquad\qquad $\tilde D_{\mathrm{\texttt Frob}} := {\bf 0}$
\item \qquad\qquad $\tilde D_s := {\bf 0}$
\item \ {\bf for} \ \,$s$\, \ {\bf from} \ \,$2u$\, \ {\bf to} \ \,$c-1$\, \ {\bf do}
\item \qquad {\bf if} \ \,$s\neq m$\, \ {\bf then}
\item \qquad\qquad $\tilde r := \tilde r-1$
\item \qquad\qquad $\mathbf{RGD}_\gamma\big(\tilde D, \ m, \ u, \ m+s+1, \ c-1, \ \tilde r\big)$
\item \qquad\qquad $\tilde D_s := {\bf 0}$
\end{enumerate}
\hrule
\hrule
}


\section{\large Counting the numerical semigroups of a given genus}

Let $n_g$ stand for the number of numerical semigroups of genus $g$.
It was conjectured in \cite{Br:fibonacci} to satisfy $n_{g+2} \,\geq\, n_{g+1}+n_{g}$
and to behave asymptotically like the Fibonacci sequence.
The latter was settled in \cite{Zhai}. The inequality has been proved in \cite{gapsets}
for the large subset of semigroups whose conductor is not over three times the multiplicity.
The weaker conjecture \,$n_{g+1}\geq n_g$,\, stated in~\cite{Segovia},
remains also open for the general case.
The development of fast algorithms for \mbox{computing~$n_g$\,}
is motivated to a great extent by these conjectures.
See~\cite{Kaplan:survey} for a nice survey on related results.

\vspace{1truemm}

We applied the RGD algorithm to the computation of \,$n_{\gamma}$ for a \mbox{given~{$\gamma\geq 4$}\,} 
through an implementation in \,\texttt{C}. The code is available at \cite{RGDc}.
Although it follows the pseudocode in the previous section quite faithfully, 
some due modifications have been made.

\vspace{1truemm}

To begin with, the $\mathbf{handle}$ function is now pointless and must be ignored throughout,
because we are not interested in the semigroups themselves but in the amount of them at level~$\gamma$. 
In particular, the first six and the last three lines of the pseudocode are omitted: obviously,
the subtrees ${\mathscr T}_2$ and ${\mathscr T}_{\gamma+1}$
have both only one node at that level. Also, it follows from Lemma~\ref{ordinary} that the number of 
semigroups of genus~$\gamma$ in the subtrees ${\mathscr T}_{\gamma}$ and~${\mathscr T}_{\gamma-1}$ is, 
respectively,\\[-5pt]
$$\gamma-1 \qquad\qquad \mathrm{and} \qquad\qquad \gamma-2 \,+\, \sum_{k=1}^{\gamma-4}k.$$
Thus, the last two steps of the loop at line 7 can be skipped by initializing the counter $n_{\gamma}$ properly. 
Moreover, for each multiplicity \,{$3\leq m<\gamma-1$}, there is only need to explore the subtree \,${\mathscr T}_m$ 
up to level~\,$\gamma-2$\, instead of~\,$\gamma$, and then add to~\,$n_{\gamma}$ the number of right generators of 
the children of the nodes at that level.
Of course, this remark concerns the implementation of the $\mathbf{RGD}_\gamma$ function too.
In regard to this point, let us also mention the following adjustments and shortcuts 
in the code:

\begin{itemize}\itemsep=1truemm

\item[·] In order to reduce the number of comparisons, lines~11 to~20
are rearranged into two possible blocks depending on whether \,$2m<\gamma$\, or not.
In the first case, the structure of the block is essentially the same as in the pseudocode:
the loop runs up to \,$u=m-1$, and then the $\mathbf{pseudo}_\gamma$ function is called
with the parameters at line~20\, for \,$c=2m$. In the second case,
the loop runs up to \,$u=\gamma-m-1$, and then we apply the counting that is explained in the next item.

\item[·] As a consequence of Lemma~\ref{threshold} and Lemma~\ref{pseudo}, 
the number of grandchildren, say \,{\texttt{gc}$(\Lambda_u)$},
of the pseudo-ordinary semigroup of jump \,$u$\, in \,${\mathscr T}_m$ can be easily computed.
Indeed, the \,$m-1$\, children of $\Lambda_u$ split into two groups:
the first~\,$u$\, or \,$u-1$\, siblings, depending on whether \,$2u\leq m$\, or not,
do have a new right generator, whereas the others do not. Thus,
\begin{center}
${\mathrm{\texttt{gc}}(\Lambda_u)} \ = \ \left\{
\begin{array}{ll}
\mathrm{\texttt{gc}}^*(m) \ + \ u & \quad \mathrm{if} \quad 2u\leq m,\\[5pt]
\mathrm{\texttt{gc}}^*(m) \ + \ u - 1 & \quad \mathrm{otherwise,}
\end{array}\right.$
\end{center}
where \,\texttt{gc}$^*(m)$ is the sum of the number of right generators of the children of $\Lambda_u$ 
that are inherited from their parent, namely
$${\mathrm{\texttt{gc}}^*(m) \ = \ (m-2) \,+\, (m-3) \,+\, \cdots\cdots \,+\, 1 \ = \ \dfrac{\,(m-1)(m-2)\,}{2}\,\cdot}$$
This formula is used in the code for \,$u=\gamma-m$\, whenever this difference is at most~$m$.

\item[·] We drop both \,$\gamma$\, and \,$g$\, as input parameters in the $\mathbf{RGD}_\gamma$ function
throughout the recursion, and the difference \,$\gamma-2-g$\, is stored and updated instead.

\item[·] The last iteration of the loop at line 20 of $\mathbf{RGD}_\gamma$ can be skipped.
\end{itemize}

\vspace{1truemm}

The code in \,\texttt{C}\, showed that the algorithm turns out to perform faster when the checking loop 
in the $\mathbf{RGD}_\gamma$ function is implemented in descending order. 
Thus, lines 8 to 11 are replaced with the following:
\begin{center}
{\footnotesize
\begin{enumerate}\itemsep0.9truemm
\item[] \qquad  $\ell:=\lfloor s/2\rfloor$
\item[] \qquad {\bf while} \ \,$\ell\geq u$\, 
\ {\bf and} \ $\big(\tilde D_{\ell}\neq{\bf 0}\,$ \ \,{\bf or}\, \ $\tilde D_{s-\ell}\neq{\bf 0}\big)$
\ {\bf do} \ \,$\ell:=\ell - 1$
\item[] \qquad  {\bf if} \ \,$\ell<u$\, \ {\bf then}
\end{enumerate}
}
\end{center}

\vspace{1truemm}

Lastly, we should add that the loops in the $\mathbf{pseudo}_\gamma$ function 
are rearranged in the implementation to avoid multiple repetitions 
of the check at lines~4 and~12. Thus, depending on whether \,$2u\leq m$\, or not, one of two possible blocks 
of three loops is performed instead. The resulting code is less compact but quite more efficient.

\vspace{5truemm}

The RGD algorithm is based on a much simpler concept than the \emph{seeds algorithm} in \cite{seeds}. 
Although the pseudocode is not so short and concise, it is more suitable
to implement in any programming language, 
since it does not require any bitwise operations through long integer data types. 
Most importantly, it turns out to be computationally much faster.
The following table displays the running times in seconds to compute $n_\gamma$ for \,$35\leq \gamma\leq 45$.
For the new algorithm, they are more than five times shorter than those which we obtain through 
a recursive implementation in \,\texttt{C}\, of the seeds algorithm, and this improvement grows with the genus.
The computations were made on a machine equipped with an 
{\footnotesize\texttt{Intel\textregistered\,Core\texttrademark\,i7-920\,}} CPU running at 
{\footnotesize\texttt{2.67GHz}}.

\vspace{2truemm}

\begin{center}
{\footnotesize
\begin{tabular}{c||c|c|c|c|c|c|c|c|c|c|c|}
$\gamma$ & $\bf 35$ & $\bf 36$ & $\bf 37$ & $\bf 38$ & $\bf 39$ & $\bf 40$ & $\bf 41$ & $\bf 42$
& $\bf 43$ & $\bf 44$ & $\bf 45$ \\ \hline\hline
& & & & & & & & & & & \\[-10pt]
\textbf{seeds}  
& 9.4
& 15.6
& 25.9
& 42.9
& 71.1
& 118.1
& 195.7
& 323
& 536
& 886
& 1465 \\ \hline
& & & & & & & & & & & \\[-10pt]
{\bf RGD}  
& 1.7
& 2.8
& 4.5
& 7.2
& 11.9
& 19.1
& 30.8
& 50
& 81
& 131
& 211 \\ \hline 
\end{tabular}
}
\end{center}

\vspace{5truemm}

The computational problem of exploring the semigroup tree 
to produce the sequence of integers \,$n_g$ is tackled in \cite{FromentinHivert}
in a most efficient way. The authors present a depth first search algorithm that is based on the decomposition numbers
of a numerical semigroup, and then they optimize it 
at a highly technical level by using SIMD methods, 
parallel branch exploration of the tree, and some implementation tricks, 
such as a partial derecursivation by means of a stack.
The source code, which is available at \cite{FH-code}, is written in \,\texttt{Cilk++} \cite{cilk}, 
an extension to the \,\texttt{C++} language that is designed for multithreaded parallel computing.

\vspace{1truemm}

We adapted our code to \,\texttt{Cilk++} \cite{RGDcilk}
in order to parallelize the exploration of the semigroup tree.
This is performed at a double level that is induced by the very structure of the RGD algorithm.
The first level corresponds to the trees~\,${\mathscr T}_m$ for~\,$m\geq 3$.
The second one, to the subtrees of \,${\mathscr T}_m$ that are rooted at the pseudo-ordinary semigroups, 
when removing the edges of the path joining these nodes.
The implementation in \,\texttt{C}\, of the algorithm is adjusted so as to apply
the \,\texttt{cilk\_for}\, command to the loops corresponding to lines~7 and~12 of the pseudocode.
The following table displays the seconds that were needed
to compute~\,$n_\gamma$\, for \mbox{\,$45\leq \gamma\leq 55$\,} on the same machine as above, 
but with the eight cores of its CPU running one thread each simultaneously.
Notice that the entry for \,$\gamma=45$\, is now more than four times smaller.

\vspace{-2truemm}

\begin{center}
{\footnotesize
\begin{tabular}{c||c|c|c|c|c|c|c|c|c|c|c|}
$\gamma$ & $\bf 45$ & $\bf 46$ & $\bf 47$ & $\bf 48$ & $\bf 49$ & $\bf 50$ & $\bf 51$ & $\bf 52$
& $\bf 53$ & $\bf 54$ & $\bf 55$ \\ \hline\hline
& & & & & & & & & & & \\[-10pt]
$\textbf{RGD}_{\mathrm{8\,t}}$
& 48.0
& 75.4
& 126.7
& 209.6
& 333.3
& 564.1
& 910.6
& 1441
& 2320
& 3876
& 6481 \\ \hline
\end{tabular}
}
\end{center}

\vspace{5truemm}

Finally, we modified the \,\texttt{Cilk++}\, code of the RGD algorithm 
so that it produces the same output as the code in \cite{FH-code}, namely the list of all integers \,$n_g$\, 
for genus \,$g$\, up to \,$\gamma$.
Then we executed both of them for \,$45\leq \gamma\leq 55$\, 
and on eight threads, still on the same machine.
The running times in seconds are displayed in the table below.
The figures seem to indicate that the improvement grows asymptotically with the genus.

\vspace{-2truemm}

\begin{center}
{\footnotesize
\begin{tabular}{c||c|c|c|c|c|c|c|c|c|c|c|}
$\gamma$ & $\bf 45$ & $\bf 46$ & $\bf 47$ & $\bf 48$ & $\bf 49$ & $\bf 50$ & $\bf 51$ & $\bf 52$
& $\bf 53$ & $\bf 54$ & $\bf 55$ \\ \hline\hline
& & & & & & & & & & & \\[-10pt]
$\textbf{F-H}_{\mathrm{8\,t}}$
& 69.0
& 112.1
& 183.0
& 296.9
& 483.5
& 822.5
& 1339
& 2175
& 3536
& 5755
& 9948 \\ \hline
& & & & & & & & & & & \\[-10pt]
$\textbf{RGD}^{\mathrm{\scriptscriptstyle\texttt{all}}}_{\mathrm{8\,t}}$
& 56.1
& 93.2
& 151.1
& 249.2
& 400.3
& 655.2
& 1076
& 1690
& 2728
& 4620
& 7266 \\ \hline 
\end{tabular}
}
\end{center}

\vspace{5truemm}

Let us complete this section with a computational result that has been obtained by 
executing the RGD algorithm on a shared server equipped with an 
{\footnotesize\texttt{AMD Ryzen 1700X\,}} CPU running at 
{\footnotesize\texttt{3.40GHz}} on sixteen threads. It goes one step beyond the data compiled at \cite{FH-code}.

\begin{thm}
There are exactly \,$2604033182682582$ numerical semigroups of genus $71$.
\end{thm}


\section*{\large Acknowledgements}
We are grateful to Omer Giménez for his useful remarks on the code, which
led us to improve the design and the implementation of the algorithm.
We also thank Ricard Gavaldà for his help on some issues of computer science,
and Enric Pons for facilitating computational resources.

\vspace{15truemm}


\bibliographystyle{plain}

\vspace*{\fill}

\begin{footnotesize}
\begin{tabular}{l}
Maria Bras-Amor\'os\\
\texttt{maria.bras\,\footnotesize{$@$}\,urv.cat}\\[5pt]
Departament d'Enginyeria Inform\`atica i  Matem\`atiques  \\
Universitat Rovira i Virgili\\
Avinguda dels Països Catalans, 26\\
E-43007 Tarragona\\[20pt]
Julio Fern\'andez-Gonz\'alez\\
\texttt{julio.fernandez.g\,\footnotesize{$@$}\,upc.edu}\\[5pt]
Departament de Matem\`atiques  \\
Universitat Polit\`ecnica de Catalunya  \\
EPSEVG -- Avinguda V\'ictor Balaguer, 1\\
E-08800 Vilanova i la Geltr\'u\\[20pt]
\end{tabular}
\end{footnotesize}

\end{document}